\documentclass[graybox]{svmult}

\usepackage{helvet}         
\usepackage{courier}        
\usepackage{type1cm}        
\usepackage{makeidx}         
\usepackage{graphicx}        
\usepackage{multicol}        
\usepackage[bottom]{footmisc}
\usepackage{bm}
\usepackage{mathtools}

\makeindex             

\usepackage{latexsym}
\usepackage{dsfont}
\usepackage{bbm}
\usepackage{amssymb}
\usepackage{amsmath}

\newtheorem{Theorem}{Theorem}[section]
\newtheorem{Lemma}[Theorem]{Lemma}
\newtheorem{Cor}[Theorem]{Corollary}
\newtheorem{Prop}[Theorem]{Proposition}
\newtheorem{Rem}[Theorem]{Remark}

\def\cS{\mathcal{S}}

\def\fA{\mathfrak{A}}

\def\Erw{\mathbb{E}}

\def\N{\mathbb{N}}
  
\def\Prob{\mathbb{P}} 

\def\R{\mathbb{R}}

\def\Z{\mathbb{Z}}

\def\bS{\textit{\bfseries S}}

\def\bZ{\textit{\bfseries Z}}

\def\1{\vec{1}}
\def\3{{\ss}}

\def\eqdist{\stackrel{d}{=}}

\def\wh{\widehat}

\def\sg{\sigma^{>}}
\def\sgn{\sigma_{n}^{>}}
\def\sge{\sigma^{\geqslant}}

\def\sle{\sigma^{\leqslant}}
\def\slen{\sigma_{n}^{\leqslant}}

\def\tg{\tau^{>}}

\def\Mg{M^{>}}
\def\Mgn{M_{n}^{>}}

\def\Mlen{M_{n}^{\leqslant}}

\def\Sg{S^{>}}
\def\Sgn{S_{n}^{>}}

\def\Sle{S^{\leqslant}}
\def\Slen{S_{n}^{\leqslant}}

\allowdisplaybreaks[4] 

\begin{document}

\title*{Ladder epochs and ladder chain of a Markov random walk with discrete driving chain}
\titlerunning{The ladder variables of a Markov random walk}
\author{Gerold Alsmeyer}
\institute{Gerold Alsmeyer \at Inst.~Math.~Statistics, Department
of Mathematics and Computer Science, University of M\"unster,
Orl\'eans-Ring 10, D-48149 M\"unster, Germany.\at
\email{gerolda@math.uni-muenster.de}}

\maketitle

\abstract{Let $(M_{n},S_{n})_{n\ge 0}$ be a Markov random walk with positive recurrent  driving chain $(M_{n})_{n\ge 0}$ having countable state space $\cS$ and stationary distribution $\pi$. It is shown in this note that, if the dual sequence $({}^{\#}M_{n},{}^{\#}S_{n})_{n\ge 0}$ is positive divergent, i.e. ${}^{\#}S_{n}\to\infty$ a.s., then the strictly ascending ladder epochs $\sgn$ of $(M_{n},S_{n})_{n\ge 0}$ (see \eqref{eq:def sgn}) are a.s. finite and the ladder chain $(M_{\sgn})_{n\ge 0}$ is positive recurrent on some $\cS^{>}\subset\cS$. We also provide simple expressions for its stationary distribution $\pi^{>}$, an extension of the result to the case when $(M_{n})_{n\ge 0}$ is null recurrent, and a counterexample that demonstrates that ${}^{\#}S_{n}\to\infty$ a.s. does not necessarily entail $S_{n}\to\infty$ a.s., but rather $\limsup_{n\to\infty}S_{n}=\infty$ a.s. only. Our arguments are based on Palm duality theory, coupling and the Wiener-Hopf factorization for Markov random walks with discrete driving chain.}

\bigskip

{\noindent \textbf{AMS 2000 subject classifications:}
60J10 (60K15) \ }

{\noindent \textbf{Keywords:} Markov-modulated sequence, Markov random walk, discrete Markov chain, ladder variables, ladder chain, stationary distribution, coupling, Palm-duality, Wiener-Hopf factorization}

\section{Introduction}
Let $M=(M_{n})_{n\ge 0}$ be a positive recurrent Markov chain, defined on a probability space $(\Omega,\fA,\Prob)$, with at most countable state space $\cS$, transition matrix $P=(p_{ij})_{i,j\in\cS}$ and unique stationary distribution $\pi=(\pi_{i})_{i\in\cS}$. Further, let $(X_{n})_{n\ge 1}$ be a sequence of real-valued random variables which are conditionally independent given $M$ and
$$ \Prob(X_{n}\in\cdot|M)\ =\ \Prob(X_{n}\in\cdot|M_{n-1},M_{n})\ =:\ F_{M_{n-1},M_{n}}\quad\text{a.s.} $$
for all $n\ge 1$. Equivalently, $(M_{n},X_{n})_{n\ge 1}$, called \emph{Markov-modulated sequence}, forms a temporally homogeneous Markov chain on $\cS\times\R$ with a special transition kernel, namely
\begin{align}
\begin{split}\label{eq:transition kernel Q}
Q((i,x),\{j\}\times (-\infty,t])\ &:=\ \Prob(M_{n+1}=j,\,X_{n+1}\le t|M_{n}=i,X_{n}=x)\\
&\ =\ \Prob(M_{n+1}=j,X_{n+1}\le t|M_{n}=i)\\
&\ =\ p_{ij}\,F_{ij}(t)
\end{split}
\end{align}
for all $i,j\in\cS$, $t,x\in\R$ and $n\ge 1$. 

\vspace{.1cm}
As usual, we write $\Prob_{i}$ for $\Prob(\cdot|M_{0}=i)$, $\Erw_{i}$ for expectations with respect to $\Prob_{i}$, and put $\Prob_{\lambda}:=\sum_{i\in\cS}\lambda_{i}\Prob_{i}$ for any measure $\lambda=(\lambda_{i})_{i\in\cS}$ on $\cS$. Under $\Prob_{\pi}$, $(M_{n},X_{n})_{n\ge 1}$ forms a stationary sequence and can therefore be extended to a doubly infinite sequence $(M_{n},X_{n})_{n\in\Z}$. Note that ``$\Prob_{\pi}$-a.s.'', also just stated as ``a.s.'' hereafter, means $\Prob_{i}$-a.s. for all $i\in\cS$ because all $\pi_{i}$ are positive. In the doubly infinite setup, we further use $\Prob_{i,x}$ for $\Prob(\cdot|M_{0}=i,X_{0}=x)$ and let $\Prob_{\nu}$ have the obvious meaning for a measure on $\cS\times\R$. Finally, the stationary distribution of $(M_{n},X_{n})_{n\in\Z}$ is denoted as $\xi$.

\vspace{.1cm}
Under the stated assumptions, the additive sequence $(S_{n})_{n\ge 0}$, defined by $S_{0}:=0$ and $S_{n}:=\sum_{k=1}^{n}X_{k}$ for $n\ge 1$, as well as its bivariate extension $(M_{n},S_{n})_{n\ge 0}$ are called a \emph{Markov random walk (MRW)} or \emph{Markov-additive process} and $M$ its \emph{(discrete) driving chain}. If $(S_{n})_{n\ge 0}$ is \emph{positive divergent}, i.e.
\begin{equation}\label{eq:positive divergence}
\lim_{n\to\infty}S_{n}\ =\ \infty\quad\text{a.s.},
\end{equation}
then the associated \emph{(strictly ascending) ladder height process} may be defined as its maximal increasing subsequence $(S_{\sgn})_{n\ge 0}$ with $\sg_{0}\equiv 0$. Here "maximal" means that any other subsequence
$(S_{\eta_{n}})_{n\ge 0}$ with positive increments and $\eta_{0}\equiv 0$ has 
a lower sampling rate $(\sgn\le\eta_{n}$ for all $n\ge 1$). Formally, we have
for $n\ge 1$
\begin{equation}\label{eq:def sgn}
\sgn\ :=\ \inf\{k>\sg_{n-1}:S_{k}>S_{\sg_{n-1}}\}
\end{equation}
with the usual convention that this is $\infty$ if $\sg_{n-1}=\infty$ or the stopping condition is never met. Using the strong Markov property, one can easily verify that
$$ (\Mgn,\Sgn)_{n\ge 0}\ :=\ (M_{\sgn},S_{\sgn})_{n\ge 0}\quad\text{and}\quad (\Mgn,\sgn)_{n\ge 0} $$
form again MRW's. Their common driving chain $(\Mgn)_{n\ge 0}$ is called \emph{ladder chain} hereafter. The main purpose of this note is to show how Palm calculus and Wiener-Hopf factorization may be employed in an elegant way to derive that the ladder chain is again positive recurrent and to provide information on its stationary distribution $\pi^{>}$, say, in terms of $\pi$. In the case when the stationary drift $\mu:=\Erw_{\pi}X_{1}$ exists and is positive and thus $n^{-1}S_{n}\to\mu$ a.s., in particular \eqref{eq:positive divergence} holds true, the positive recurrence of $(\Mgn)_{n\ge 0}$ along with other properties of $(\Mgn,\Sgn,\sgn)_{n\ge 0}$ has already been proved in \cite{Alsmeyer:00} even allowing for uncountable state space $\cS$. On the other hand, the arguments given there are rather technical, owing to the more delicate renewal structure of a general positive Harris chain on a continuous state space as opposed to a discrete Markov chain. To keep the amount of technicalities at a minimum has been the main reason to restrict ourselves here to the discrete setting.

\vspace{.1cm}
It should be clear that recurrence of the ladder chain and related properties form an important ingredient when dealing with Markov renewal theory or, more generally, fluctuation-theoretic properties of the MRW $(M_{n},S_{n})_{n\ge 0}$. Namely, it allows to identify a subsequence $(S_{\tg_{n}(s)})_{n\ge 0}$ of $(\Sgn)_{n\ge 0}$ and thus of $(S_{n})_{n\ge 0}$ which is an ordinary renewal process (a RW with positive increments) because $M_{\tg_{n}(s)}=s$ for all $n\ge 1$ and some $s\in\cS$. Such embeddings are fundamental when attempting to derive results of the afore-mentioned kind for MRW's by drawing on known results for ordinary RW's or renewal processes. For the case when $(M_{n},S_{n})_{n\ge 0}$ has positive stationary drift, this has recently been demonstrated in \cite{Alsmeyer:14} by showing that all fundamental Markov renewal theorems can be deduced with the help of such embeddings and the use of classical renewal theory.

\section{Main results}\label{sec:main results}

In order to present our main results, we first need to return to the doubly infinite sequence $(M_{n},X_{n})_{n\in\Z}$ under $\Prob_{\xi}$ and define the associated doubly infinite random walk $(S_{n})_{n\in\Z}$ via
\begin{equation}\label{eq:def doubly infinite RW}
S_{n}\ :=\ 
\begin{cases}
\hfill\sum_{i=1}^{n}X_{i},&\text{if }n\ge 1,\\
\hfill 0,&\text{if }n=0,\\
\hfill-\sum_{i=n+1}^{0}X_{i},&\text{if }n\le -1,
\end{cases}
\end{equation}
thus $S_{n}=S_{n-1}+X_{n}$ for all $n\in\Z$. We note that the forward sequence $(M_{n},X_{n})_{n\in\Z}$ is a stationary Markov chain with transition kernel $Q$ given by \eqref{eq:transition kernel Q}, while the backward sequence $({}^{\#}M_{n},{}^{\#}X_{n})_{n\in\Z}:=(M_{-n},X_{-n+1})_{n\ge 0}$ is a stationary Markov chain with the dual kernel ${}^{\#}Q$ given by
\begin{equation}\label{eq:dual kernel}
{}^{\#}Q((i,x),\{j\}\times (-\infty,t])\ =\ \frac{\pi_{j}p_{ji}}{\pi_{i}}\,F_{ji}(t)
\end{equation}
for $i,j\in\cS$ and $x,t\in\R$. We call $({}^{\#}M_{n},{}^{\#}X_{n})_{n\in\Z}$ and $({}^{\#}M_{n})_{n\in\Z}$ the dual of $(M_{n},X_{n})_{n\in\Z}$ and $(M_{n})_{n\in\Z}$, respectively. Accordingly, the MRW $({}^{\#}M_{n},{}^{\#}S_{n})_{n\in\Z}$ with ${}^{\#}S_{n}$ having the obvious meaning is called the dual of $(M_{n},S_{n})_{n\ge 0}$. Note that ${}^{\#}S_{n}=-S_{-n}$ for $n\in\Z$.
 
\vspace{.1cm}
If $\mu=\Erw_{\pi}X_{0}>0$, then Birkhoff's ergodic theorem implies \eqref{eq:positive divergence} as well as the positive divergence of the dual walk, i.e.
\begin{equation}\label{eq:left positive divergence}
\lim_{n\to\infty}{}^{\#}S_{n}\ =\ \lim_{n\to\infty}-S_{-n}\ =\ \infty\quad\text{a.s.}
\end{equation}
However, the last assertion may fail if only \eqref{eq:positive divergence} holds, see Section \ref{sec:IPFC} at the end of the paper for a counterexample. If \eqref{eq:left positive divergence} is assumed, then, with probability one, there is a doubly infinite sequence $(\sigma_{n})_{n\in\Z}$ of ladder epochs determined through
$...<\sigma_{-1}<\sigma_{0}\le 0<\sigma_{1}<\sigma_{2}<...$ and $S_{\sigma_{n}}
>\sup_{j<\sigma_{n}}S_{j}$ for all $n\in\Z$. In particular,
\begin{align*}
\sigma_{1}\ &:=\ \inf\{k\ge 1:S_{k}>\sup_{j<k}S_{j}\},\\
\sigma_{0}\ &:=\ \sup\{k\le 0:S_{k}>\sup_{j<k}S_{j}\}.
\end{align*}
The reader should notice that $(\sigma_{n})_{n\ge 1}$ and $(\sgn)_{n\ge 1}$ are generally different although these sequences share the same recursive structure (see \eqref{eq:def sgn}):
$$ \sigma_{n}\ =\ \inf\{k>\sigma_{n-1}:S_{k}-S_{\sigma_{n-1}}>0\}. $$

Our main result can be viewed as a specialization of a similar result stated (without proof) by Lalley \cite[Section 4B]{Lalley:86} for general random walks with integrable stationary increments.

\begin{Theorem}\label{thm:first result}
Let $(M_{n},S_{n})_{n\ge 0}$ be a MRW having positive recurrent discrete driving chain $(M_{n})_{n\ge 0}$ with state space $\cS$, transition matrix $P=(p_{ij})_{i,j\in\cS}$ and stationary distribution $\pi=(\pi_{i})_{i\in\cS}$. Suppose that the dual $({}^{\#}S_{n})_{n\ge 0}$ is positive divergent in the sense of \eqref{eq:left positive divergence}. Then the ladder chain $\Mg=(\Mgn)_{n\ge 0}$ possesses the unique stationary distribution $\pi^{>}=(\pi_{i}^{>})_{i\in\cS}$, defined by
\begin{align}\label{eq:pig first identity}
\pi_{i}^{>}\ :=\ \frac{1}{c}\,\Erw_{\xi}\left({1\over\sigma_{1}-\sigma_{0}}\1_{\{M_{\sigma_{0}}=i\}}\right)\ =\ \frac{1}{c}\,\Prob_{\xi}(M_{0}=i,\sigma_{0}=0),
\end{align}
and has $\pi$-density $f(i)=c^{-1}\,\Prob_{\xi}(\sigma_{0}=0|M_{0}=i)$ for $i\in\cS$. Here $\xi$ equals the stationary law of $(M_{n},X_{n})_{n\in\Z}$ and
$$ c\ :=\ \Erw_{\xi}\left({1\over\sigma_{1}-\sigma_{0}}\right)\ =\ \Prob_{\xi}(\sigma_{0}=0)\ \in\  (0,1]. $$
Moreover, the ladder chain is positive recurrent on $\cS^{>}:=\{i\in\cS:\pi_{i}^{>}>0\}$, and
\begin{equation}\label{Mg ergodic}
\Prob_{i}\left(\tg(\cS^{>})<\infty\right)=1
\end{equation}
for all $i\in\cS$, where $\tg(\cS^{>}):=\inf\{n\ge 1:\Mgn\in\cS^{>}\}$.
\end{Theorem}

\begin{Rem}\label{rem1:condition (6)}\rm
As a consequence of Theorem \ref{thm:first result}, we see that \eqref{eq:left positive divergence} entails the existence of an ordinary two-sided subsequence $(S_{\tau_{n}})_{n\in\Z}$ having positive iid increments (choose the $\tau_{n}$ as those ladder epochs with $\tau_{0}\le 0<\tau_{1}$ having further $M_{\tau_{n}}=s$ for some fixed $s\in\cS^{>}$) and therefore $\limsup_{n\to\infty}S_{n}=\infty$ a.s. which is weaker than positive divergence of $(S_{n})_{n\ge 0}$.
\end{Rem}

\begin{Rem}\label{rem2:condition (6)}\rm
Let us further point out that \eqref{eq:left positive divergence} is sufficient but not necessary for the positive recurrence of the ladder chain on some $\cS^{>}\subset\cS$. As an example consider a MRW $(M_{n},S_{n})_{n\ge 0}$ with positive recurrent driving chain such that, for some unique $s\in\cS$, the $F_{is}$, $i\in\cS$, are concentrated on $\R_{>}$, while all other $F_{ij}$ are concentrated on $\R_{<}$. It is not difficult to further arrange for stationary drift $\mu=\Erw_{\pi}X_{1}=0$ and nontriviality in the sense that $\Prob_{\pi}(X_{1}=g(M_{1})-g(M_{0}))<1$ for any function $g:\cS\to\R$, giving
$$ \liminf_{n\to\infty}S_{n}\ =\ -\infty\quad\text{and}\quad\limsup_{n\to\infty}S_{n}\ =\ \infty\quad\text{a.s.} $$
(see \cite[Thms. 2 and 3]{Alsmeyer:01} or \cite{AlsBuck:16}) and thus the same for its dual $({}^{\#}S_{n})_{n\ge 0}$. On the other hand, since $X_{\sgn}$ must be positive for each $n$, we find that $\Mg_{n}=s$ for all $n\ge 1$ and so the ladder chain is trivially positive recurrent on $\cS^{>}=\{s\}$.
\end{Rem}

The following simple corollary provides an alternative definition of the $\pi_{i}^{>}$ with the help of the weakly descending ladder epochs of the dual MRW $({}^{\#}M_{n},{}^{\#}S_{n})_{n\ge 0}$, defined by ${}^{\#}\sle_{0}:=0$ and
$$ {}^{\#}\slen\ :=\ \inf\{k>{}^{\#}\sle_{n-1}:{}^{\#}S_{k}\le {}^{\#}S_{{}^{\#}\sle_{n-1}}\} $$
for $n\ge 1$, where $\inf\emptyset:=\infty$ and ${}^{\#}\sle:={}^{\#}\sle_{1}$.

\begin{Cor}\label{cor:first result}
Under the assumptions of Theorem \ref{thm:first result}, it further holds that
\begin{equation}\label{eq:pig second identity}
\pi_{i}^{>}\ =\ \frac{1}{c}\,\pi_{i}\,\Prob_{i}\left({}^{\#}\sle=\infty\right)\ =\ \frac{1}{c}\,\Prob_{\pi}\left({}^{\#}M_{0}^{\leqslant}=i,{}^{\#}\sle=\infty\right)
\end{equation}
for all $i\in\cS$, giving
$$ c\ :=\ \Prob_{\pi}\left({}^{\#}\sle=\infty\right) $$
and 
\begin{equation}\label{eq:form of S^>}
\cS^{>}\ =\ \left\{i\in\cS:\Prob_{i}\left({}^{\#}\sle=\infty\right)>0\right\}.
\end{equation}
\end{Cor}

\begin{proof}
It suffices to point out that
\begin{align*}
\Prob_{\xi}(M_{0}=i,\sigma_{0}=0)\ &=\ \Prob_{\xi}\left(M_{0}=i,\,\max_{n\ge 1}S_{-n}<0=S_{0}\right)\\
&=\ \Prob_{\xi}\left({}^{\#}M_{0}=i,\,\min_{n\ge 1}{}^{\#}S_{n}>0\right)\\
&=\ \pi_{i}\,\Prob_{i}\left({}^{\#}\sle=\infty\right)
\end{align*}
for all $i\in\cS$.\qed
\end{proof}

Instead of the second equality in \eqref{eq:pig first identity}, we will actually verify the more general one
\begin{equation}\label{eq:pig first identity generalized}
\frac{1}{cm}\,\Prob_{\xi}(M_{\sigma_{0}}=i,\sigma_{1}-\sigma_{0}=m)\ =\ \frac{1}{c}\,\Prob_{\xi}(M_{0}=i,\sigma_{0}=0,\sigma_{1}=m)
\end{equation}
for all $(i,m)\in\cS\times\N$ in the proof of Theorem \ref{thm:first result} below. Using this identity and the obvious fact that $(\Mgn,\sg_{n+1}-\sgn)_{n\ge 0}$ is a stationary Markov chain under $\Prob_{\pi^{>}}$, we are led to the following corollary.

\begin{Cor}\label{cor:ladder epochs}
Under the assumptions of Theorem \ref{thm:first result}, it further holds that the law $\nu=(\nu_{i,m})_{(i,m)\in\cS\times\N}$ of $(\Mgn,\sg_{n+1}-\sgn)$ under $\Prob_{\pi^{>}}$ is given by
\begin{align}
\begin{split}\label{eq:def on nu}
\nu_{i,m}\ &=\ \frac{1}{cm}\,\Prob_{\xi}(M_{\sigma_{0}}=i,\sigma_{1}-\sigma_{0}=m)
\end{split}
\end{align}
In particular,
\begin{align}
\Prob_{\pi^{>}}(\sg=m)\ &=\ \sum_{i\in\cS}\nu_{i,m}\ =\ \frac{\Prob_{\xi}(\sigma_{1}-\sigma_{0}=m)}{cm}\label{eq:def of nu}
\shortintertext{and}
\Erw_{\pi^{>}}\sg\ &=\ \frac{1}{\Prob_{\pi}\left({}^{\#}\sle=\infty\right)}.\label{eq:Esg}
\end{align}
\end{Cor}

\begin{proof}
We have that
\begin{align*}
\nu_{i,m}\ &=\ \Prob_{\pi^{>}}(\Mg_{0}=i,\sg=m)\\
&=\ \pi_{i}^{>}\,\Prob_{i}(\sg=m)\\
&=\ \frac{1}{c}\Prob_{\xi}(M_{0}=i,\sigma_{0}=0)\,\Prob_{i}(\sg=m)\\
&=\ \frac{1}{c}\Prob_{\xi}(M_{0}=i,\sigma_{0}=0,\sigma_{1}=m)\\
&=\frac{1}{cm}\,\Prob_{\xi}(M_{\sigma_{0}}=i,\sigma_{1}-\sigma_{0}=m)
\end{align*}
as claimed. The remaining assertions are now obvious.\qed
\end{proof}

Since $\pi_{i}^{>}\,\Erw_{i}\sg\le\Erw_{\pi^{>}}\sg$ and $\pi_{i}^{>}=c^{-1}\pi_{i}\,\Prob_{i}({}^{\#}\sle=\infty)$ for any $i\in\cS$, identity \eqref{eq:Esg} further provides us with
\begin{equation}
\Erw_{i}\sg\ \le\ \frac{1}{\pi_{i}\,\Prob_{i}\left({}^{\#}\sle=\infty\right)}
\end{equation}
for all $i\in\cS$ with the right-hand side being finite only for $i\in\cS^{>}$.

\vspace{.2cm}
A direct proof of Theorem \ref{thm:first result} will be presented in the following section, but its most critical part, namely the existence of $\pi^{>}$, could also be deduced by drawing on a more general duality result from Palm calculus as described in the monographies by Thorisson \cite[Ch.~8]{Thorisson:00} and Sigman \cite{Sigman:95}. We refer to Remark \ref{rem:Palm calculus} for a sketch of details. Yet another proof via the Wiener-Hopf factorization for MRW's is briefly shown in Section \ref{sec:WHF approach}, followed by a short treatment of the case when the driving chain is null recurrent in Section \ref{sec:null recurrent case}, see Theorem \ref{thm:second result} there.

\section{Proof of Theorem \ref{thm:first result}}\label{sec:proof first theorem}

We start by showing equality of the two expressions for $\pi_{i}^{>}$ in \eqref{eq:pig first identity}. We have
\begin{align*}
\Erw_{\xi}\Bigg({1\over\sigma_{1}-\sigma_{0}}&\1_{\{M_{\sigma_{0}}=i\}}\Bigg)\\
&=\ \sum_{k\ge 0}\,\sum_{l>k}\frac{1}{l}\,\Prob_{\xi}(M_{-k}=i,\sigma_{0}=-k,\sigma_{1}-\sigma_{0}=l)\\
&=\ \sum_{k\ge 0}\,\sum_{l>k}\frac{1}{l}\,\Prob_{\xi}(M_{0}=i,\sigma_{0}=0,\sigma_{1}-\sigma_{0}=l)\\
&=\ \sum_{l\ge 1}\Prob_{\xi}(M_{0}=i,\sigma_{0}=0,\sigma_{1}-\sigma_{0}=l)\sum_{k=0}^{l-1}\frac{1}{l}\\
&=\ \sum_{l\ge 1}\Prob_{\xi}(M_{0}=i,\sigma_{0}=0,\sigma_{1}-\sigma_{0}=l)\\
&=\ \Prob_{\xi}(M_{0}=i,\sigma_{0}=0),
\end{align*}
where the third line follows from the stationarity of $(M_{n},X_{n})_{n\in\Z}$ (under $\Prob_{\xi}$) which provides us with
\begin{align*}
&\{M_{-k}=i,\sigma_{0}=-k,\sigma_{1}-\sigma_{0}=l\}\\
&=\ \left\{M_{-k}=i,\,S_{-k}>\max_{j>k}S_{-j},\,\max_{1\le j<l}S_{-k+j}\le S_{-k},\,S_{-k+l}>\max_{0\le j<l}S_{-k+j}\right\}\\
&=\ \left\{M_{-k}=i,\,\min_{j>k}\sum_{n=-j}^{-k}X_{n}>0,\,\max_{1\le j<l}\sum_{n=-k+1}^{-k+j}X_{n}\le 0,\,\min_{1\le j\le l}\sum_{n=-k+j}^{-k+l}X_{n}>0\right\}\\
&\eqdist\  \left\{M_{0}=i,\,\min_{j>0}\sum_{n=-j}^{0}X_{n}>0,\,\max_{1\le j<l}\sum_{n=1}^{j}X_{n}\le 0,\,\min_{1\le j\le l}\sum_{n=j}^{l}X_{n}>0\right\}\\
&=\ \{M_{0}=i,\sigma_{0}=0,\sigma_{1}-\sigma_{0}=l\}
\end{align*}
for all $k\ge 0$. Here $A\eqdist B$ means that $\1_{A}$ and $\1_{B}$ have the same law.

\vspace{.1cm}
Since
$$ \pi_{i}^{>}\ =\ \frac{1}{c}\,\Prob_{\xi}(\sigma_{0}=0|M_{0}=i)\,\Prob_{\xi}(M_{0}=i)\ =\ \frac{1}{c}\,\Prob_{i}(\sigma_{0}=0)\,\pi_{i} $$
we also obtain the asserted form of the $\pi$-density von $\pi^{>}$.

\vspace{.1cm}
The next step is to verify that $\pi^{>}$ defines a stationary distribution for $(\Mg_{n})_{n\in\Z}$. By a similar argument as before, we obtain
\begin{align*}
\Prob_{\pi^{>}}(\Mg_{1}=j)\ &=\ \sum_{i\in\cS}\pi_{i}^{>}\,\Prob_{i}(\Mg_{1}=j)\\
&=\ \frac{1}{c}\,\sum_{i\in\cS}\Prob_{\xi}(\sigma_{0}=0,M_{0}=i)\,\Prob_{i}(\Mg_{1}=j)\\
&=\ \frac{1}{c}\,\sum_{i\in\cS}\sum_{n\ge 1}\Prob_{\xi}(\sigma_{0}=0,\sigma_{1}=n,M_{0}=i,M_{n}=j)\\
&=\ \frac{1}{c}\,\sum_{i\in\cS}\sum_{n\ge 1}\Prob_{\xi}(\sigma_{-1}=-n,\sigma_{0}=0,M_{-n}=i,M_{0}=j)\\
&=\ \frac{1}{c}\,\sum_{i\in\cS}\Prob_{\xi}(\Mg_{-1}=i,\sigma_{0}=0,M_{0}=j)\\
&=\ \frac{1}{c}\,\Prob_{\xi}(\sigma_{0}=0,M_{0}=j)\ =\ \pi_{j}^{>}
\end{align*}
for all $j\in\cS$, and this yields the desired result, for $(\Mg_{n})_{n\in\Z}$ is a Markov chain under $\Prob_{\xi}$.

\vspace{.1cm}
Clearly, all states in $\cS^{>}$ are positive recurrent for the ladder chain. A coupling argument will now be used to establish the remaining assertions including uniqueness of $\pi^{>}$ as a stationary distribution of $(\Mgn)_{n\ge 0}$. On a possibly enlarged probability space with underlying probability measure $\Prob$, let $(M_{n}',X_{n}')_{n\ge 0}$ and $(M_{n}'',X_{n}'')_{n\ge 0}$ be two Markov-modulated sequences with the same transition kernel as $(M_{n},X_{n})_{n\ge 0}$ and initial conditions
$$ (M_{0}',M_{0}'')=(i,j)\quad\text{and}\quad X_{0}'=X_{0}''=0 $$
for arbitrarily fixed $i,j\in\cS$. As usual, the associated RW's are denoted by $(S_{n}')_{n\ge 0}$ and $(S_{n}'')_{n\ge 0}$, the corresponding strictly ascending ladder epochs by $\sigma_{n}'$ and $\sigma_{n}''$, respectively, where $\sigma_{0}'=\sigma_{0}'':=0$.

\vspace{.1cm}
Let $T$ be the $\Prob$-a.s.\ finite coupling time of $(M_{n}',M_{n}'')_{n\ge 0}$, thus
$$ T\ =\ \inf\{n\ge 0:M_{n}'=M_{n}''\}, $$
and put $Y':=\max_{0\le n\le T}S_{n}'$, $Y'':=\max_{0\le n\le T}S_{n}''$ and $Y:=Y'\vee Y''$.
Then the coupling process
\begin{equation*}
(\wh{M}_{n},\wh{X}_{n})\ :=\ 
\begin{cases}
\hfill (M_{n}',X_{n}'),&\text{falls }n\le T,\\
(M_{n}'',X_{n}''),&\text{falls }n>T,
\end{cases}
\end{equation*}
forms a copy of $(M_{n}',X_{n}')_{n\ge 0}$ and coincides with $(M_{n}'',X_{n}'')_{n\ge 0}$ after $T$. Moreover, $\wh{S}_{n}=S_{n}'$ for $n\le T$ and $\wh{S}_{n}=S_{n}''+(S_{T}'-S_{T}'')$ for $n>T$. Denoting by $(\wh{\sigma}_{n})_{n\ge 0}$ the sequence of strictly ascending ladder epochs of $(\wh{S}_{n})_{n\ge 0}$, the crucial observation now is that the random sets $\{\sigma_{n}'',\,n\ge 0\}$ and $\{\wh{\sigma}_{n},\,n\ge 0\}$ a.s. coincide up to finitely many elements. Namely, if
\begin{align*}
\tau\ :=\ \inf\{n:S_{\sigma_{n}''}''>Y+(S_{T}''-S_{T}')^{+}\},\\
\rho\ :=\ \inf\{n:\wh{S}_{\wh{\sigma}_{n}}>Y+(S_{T}'-S_{T}'')^{+}\},
\end{align*}
then $\sigma_{\tau}''=\wh{\sigma}_{\rho}>T$ and thus $\sigma_{\tau+n}''=\wh{\sigma}_{\rho+n}$ for all $n\ge 1$ because $(S_{n}'')_{n\ge 0}$ and $(\wh{S}_{n})_{n\ge 0}$ have the same increments after $T$.

\vspace{.1cm}
To complete the proof, put ${M_{n}''}^{>}:=M_{\sigma_{n}''}''$, $\wh{M}_{n}^{>}:=\wh{M}_{\wh{\sigma}_{n}}$ for $n\ge 0$ and notice that
$$ ({M_{n}''}^{>})_{n\ge\tau}\ =\ (\wh{M}_{n}^{>})_{n\ge\rho} $$
for any choice of $i,j\in\cS$. Consequently,
\begin{align*}
\Prob_{i}(\Mg_{n}=j\text{ i.o.})\ &=\ \Prob(\wh{M}_{n}^{>}=j\text{ i.o.})
\ =\ \Prob(\wh{M}_{\rho+n}^{>}=j\text{ i.o.})\\
&=\ \Prob({M_{\tau+n}''}\hspace{-9pt}{}^{>}=j\text{ i.o.})
\ =\ \Prob({M_{n}''}^{>}=j\text{ i.o.})\\
&=\ \Prob_{j}(\Mg_{n}=j\text{ i.o.})
\end{align*}
for all $i\in\cS$ and $j\in\cS^{>}$ which shows the irreducibility of the ladder chain on $\cS^{>}$ as well as \eqref{Mg ergodic}.\qed

\begin{Rem}\label{rem:Palm calculus}\rm
Let us briefly describe how the Theorem \ref{thm:first result} fits into the framework of Palm calculus as laid out in \cite[Ch.~8]{Thorisson:00}. Our starting point is here the stationary sequence $(M_{n},X_{\leqslant n})_{n\in\Z}$ under $\Prob_{\xi}$, where $X_{\leqslant n}:=(X_{k})_{k\le n}$. Notice that
$$ D_{m}\ :=\ S_{m}-\max_{k\le m}S_{k}\ =\ \min_{k\le m}\sum_{j=k+1}^{m}X_{j} $$
is a functional of $X_{\leqslant m}$ which equals 0 iff $m$ is a strictly ascending ladder epoch of the associated doubly infinite MRW $(M_{n},S_{n})_{n\in\Z}$. In other words, the sequence $\bS:=(\sigma_{n})_{n\in\Z}$ of ladder epochs may be viewed as the sequence of return times to $\cS\times\{0\}$ of the stationary sequence $\bZ:=(M_{n},D_{n})_{n\in\Z}$, and it is a functional of it. For $n\in\Z$, define the two-sided shift
$$ \theta^{n}\big((z_{k})_{k\in\Z},(t_{k})_{k\in\Z}\big)\ :=\ \big((z_{n+k})_{k\in\Z},(t_{n,k})_{k\in\Z}\big) $$
for $n\in\Z$, $(z_{k})_{k\in\Z}\in(\cS\times\R_{\leqslant})^{\Z}$ and strictly increasing sequences $(t_{k})_{k\in\Z}$ such that
$$ -\infty \leftarrow \ldots < t_{-1} \le 0 < t_{1} < \ldots \to \infty $$
and $(t_{n,k})_{k\in\Z}$ equals the sequence $(t_{k}+n)_{k\in\Z}$ modulo relabeling so as to have $t_{n,0}\le 0<t_{n,1}$ (see also \cite[p.~251]{Thorisson:00}). Then the above considerations imply that the sequence
$$ (\theta^{n}(\bZ,\bS\,))_{n\in\Z} $$
is $\Prob_{\pi}$-stationary, and the Palm duality theory \cite[Theorem 8.4.1]{Thorisson:00} now tells us that its cycles
$$ C_{n}\ :=\ \left(\sigma_{n+1}-\sigma_{n},(M_{k},D_{k})_{\sigma_{n}\le k<\sigma_{n+1}}\right),\quad n\in\Z, $$
and the sequence $(M_{\sigma_{n}},\sigma_{n+1}-\sigma_{n})_{n\in\Z}$ in particular are stationary under the probability measure $\Prob_{\xi}^{0}$, defined by
$$ \Prob_{\xi}^{0}(dx)\ :=\ \frac{1}{c(\sigma_{1}-\sigma_{0})}\,\Prob_{\xi}(dx) $$
with $c$ as in Theorem \ref{thm:first result} and satisfying $\Prob_{\pi}^{0}(\sigma_{0}=0)=1$. Consequently,
$$ \pi_{i}^{>}\ =\ \Prob_{\xi}^{0}(M_{0}=i)\ =\ \frac{1}{c}\,\Erw_{\xi}\left(\frac{1}{\sigma_{1}-\sigma_{0}}\1_{\{M_{\sigma_{0}}=i\}}\right),\quad i\in\cS $$
is a stationary distribution for the ladder chain as asserted in our theorem.
\end{Rem}

\section{An alternative approach via Wiener-Hopf factorization}\label{sec:WHF approach}

That $\pi^{>}$ as defined in \eqref{eq:pig second identity} forms a stationary distribution of the ladder chain, may also be derived with the help of the Wiener-Hopf factorization for MRW's as we will briefly demonstrate after recalling some necessary facts about this factorization with reference to Asmussen \cite{Asmussen:89} and \cite[p.~314ff]{Asmussen:03}. Putting $({}^{\#}\Mlen,{}^{\#}\Slen):=({}^{\#}M_{{}^{\#}\slen},{}^{\#}S_{{}^{\#}\slen})$, we define the matrices
\begin{equation*}
G := (G_{ij})_{i,j\in\cS},\ G^{>}:=\big(G_{ij}^{>})_{i,j\in\cS},\ {}^{\#}G:=\big({}^{\#}G_{ij}\big)_{i,j\in\cS},\ {}^{\#}G:=\big({}^{\#}G_{ij}\big)_{i,j\in\cS}
\end{equation*}
with measure-valued entries by
\begin{align*}
&G_{ij}\ :=\ p_{ij}\,F_{ij}\ =\ \Prob_{i}(M_{1}=j,X_{1}\in\cdot),\\
&G_{ij}^{>}\ :=\ \Prob_{i}(\Mg_{1}=j,\Sg_{1}\in\cdot,\sg<\infty),\\
{}^{\#}&G_{ij}^{\leqslant}\ :=\ \Prob_{i}\left({}^{\#}M_{1}^{\leqslant}=j,{}^{\#}\Sle_{1}\in\cdot,{}^{\#}\sle<\infty\right),\\
{}^{\star}&G_{ij}^{\leqslant}\ :=\ \frac{\pi_{j}}{\pi_{i}}\,{}^{\#}G_{ji}^{\leqslant}.
\end{align*}
The convolution of matrices $A=(A_{ij})_{i,j\in\cS},B=(B_{ij})_{i,j\in\cS}$ with measure-valued entries is defined in the usual manner by replacing ordinary multiplication with convolution of measures, thus
$A*B=\left(\sum_{k\in\cS}A_{ik}*B_{kj}\right)_{i,j\in\cS}$.
The following result, stated here for reference, provides the \emph{Wiener-Hopf factorization of a MRW with discrete driving chain} and is Theorem 4.1 in \cite{Asmussen:89}.

\begin{Prop}\label{prop:WHF for MRW}
Let $(M_{n},S_{n})_{n\ge 0}$ be a MRW with positive recurrent discrete driving chain $(M_{n})_{n\ge 0}$. Then
\begin{equation}\label{eq:WHF1}
\delta_{0}\,I-G\ =\ (\delta_{0}\,I-{}^{\star}G^{\leqslant})*(\delta_{0}\,I-G^{>})
\end{equation}
or, equivalently,
\begin{equation}\label{WHF2}
G\ =\ {}^{\star}G^{\leqslant}+G^{>}-{}^{\star}G^{\leqslant}*G^{>},
\end{equation}
where $I$ denotes the identity matrix on $\cS$.
\end{Prop}

Let us also state the following lemma in which $\|\nu\|$ denotes the total mass of a measure $\nu$.

\begin{Lemma}\label{lem:auxiliary}
Under the assumptions of Theorem \ref{thm:first result}, the matrix $\|{}^{\#}G^{\leqslant}\|:=(\|{}^{\#}G_{ij}^{\leqslant}\|)_{i,j\in\cS}$ is truly substochastic, i.e.
$$ \sum_{j\in\cS}\|{}^{\#}G_{ij}^{\leqslant}\|\ \le\ 1 $$
for all $i\in\cS$ with strict inequality for at least one $i$. Furthermore,
\begin{equation}\label{f980}
\sum_{i\in\cS}\pi_{i}\,\|{}^{\star}G_{ij}^{\leqslant}\|\ =\ \pi_{j}\,\Prob_{j}({}^{\#}\sle<\infty)\ \le\ \pi_{j}
\end{equation}
for all $j\in\cS$ with strict inequality for at least one $j$.
\end{Lemma}

\begin{proof}
By the definition of ${}^{\#}G_{\ij}^{\leqslant}$, we have
$$ \sum_{j\in\cS}\|{}^{\#}G_{\ij}^{\leqslant}\|\ =\ \sum_{j\in\cS}\Prob_{i}\left({}^{\#}M_{1}^{\leqslant}=j,{}^{\#}\sle<\infty\right)\ =\ \Prob_{i}\left({}^{\#}\sle<\infty\right)\ \le\ 1 $$
for all $i\in\cS$. Moreover, strict inequality must hold for at least one $i$, for otherwise
$\Prob_{\pi}({}^{\#}\sle<\infty)=1$ would follow and then inductively $\Prob_{\pi}({}^{\#}\slen<\infty)=1$ for all $n\in\N$, i.e.
$$ \Prob_{\pi}\left({}^{\#}S_{n}\le 0\text{ i.o.}\right)\ =\ 1, $$
which is impossible under the assumption of two-sided positive divergence.

\vspace{.1cm}
For the proof of \eqref{f980}, we note that
\begin{align*}
\sum_{i\in\cS}\pi_{i}\,\|{}^{\star}G_{\ij}^{\leqslant}\|\ &=\ \sum_{i\in\cS}\pi_{j}\,\|{}^{\#}G_{ji}^{\leqslant}\|\\
&=\ \pi_{j}\sum_{i\in\cS}\Prob_{j}({}^{\#}M_{1}^{\leqslant}=i,{}^{\#}\sle<\infty)\\
&=\ \pi_{j}\,\Prob_{j}({}^{\#}\sle<\infty),
\end{align*}
and by an analogous argument as before this value must be less than $\pi_{j}$ for at least one $j$ if two-sided positive divergence holds.\qed
\end{proof}

\begin{proof}[of the stationarity of $\pi^{>}$ given by \eqref{eq:pig second identity}]
It follows from the Wiener-Hopf factorization \eqref{eq:WHF1} that
\begin{equation}\label{f983}
I-\|G\|\ =\ \left(I-\|{}^{\star}G^{\leqslant}\|\right)\left(I-\|G^{>}\|\right).
\end{equation}
Multiplying this identity from the left with $\pi^{\top}$ (the transpose of $\pi$) and observing that
$\|G\|=P=(p_{\ij})_{i,j\in\cS}$ is the transition matrix of $M$, we infer
\begin{align*}
0\ =\ \pi^{\top}(I-P)\ =\ \pi^{\top}\left(I-\|{}^{\star}G^{\leqslant}\|\right)\left(I-\|G^{>}\|\right).
\end{align*}
By Lemma \ref{lem:auxiliary}, in particular \eqref{f980}, and the fact that all $\pi_{i}$ are positive, the nonnegative vector
$$ \pi^{\top}\left(I-\|{}^{\star}G^{\leqslant}\|\right)\ =\ \left(\pi_{i}\,\Prob_{i}\left({}^{\#}\sge=\infty\right)\right)_{i\in\cS} $$
is not identially zero and thus a proper solution to the equation $x(I-\|G^{>}\|)=0$. After normalization through $c$, it therefore forms a stationary distribution of $M^{>}$ with transition matrix $\|G^{>}\|$.\qed
\end{proof}

\section{The null recurrent case}\label{sec:null recurrent case}

It is not difficult to extend Theorem \ref{thm:first result} to the case when, ceteris paribus, the driving chain $(M_{n})_{n\ge 0}$ is null recurrent with essentially unique (up to positive scalars) stationary measure $\pi$. First of all, it should be observed that the ladder chain $(\Mgn)_{n\ge 0}$ may still be positive recurrent on some $\cS^{>}$. In fact, $(\Mgn)_{n\ge 0}$ takes only values in the set
$$ \cS^{+}\ :=\ \{s\in\cS:F_{is}(\R_{>})>0\text{ for some }i\in\cS\} $$
positive recurrence on some $\cS^{>}\subset\cS^{+}$ follows whenever $\cS^{+}$ is finite (see also Remark \ref{rem2:condition (6)}).

The following theorem is proved in essentially the same manner as Theorem \ref{thm:first result}, and we therefore restrict ourselves to some comments regarding its proof. Put
$$ \xi\ :=\ \Prob_{\pi}((M_{1},X_{1})\in\cdot), $$
which is the essentially unique stationary measure of $(M_{n},X_{n})_{n\ge 0}$. Since $\pi$ and $\xi$ have infinite mass now, we define
\begin{align*}
c\ =\ 
\begin{cases}
\Erw_{\xi}\left(\sigma_{1}-\sigma_{0}\right)^{-1}=\Prob_{\xi}(\sigma_{0}=0),&\text{if these expressions are finite},\\
\hfill 1,&\text{otherwise}.
\end{cases}
\end{align*}

\begin{Theorem}\label{thm:second result}
Let $(M_{n},S_{n})_{n\ge 0}$ be a MRW having null recurrent discrete driving chain $(M_{n})_{n\ge 0}$ with essentially unique stationary measure $\pi=(\pi_{i})_{i\in\cS}$. Suppose that the dual $({}^{\#}S_{n})_{n\ge 0}$ is positive divergent in the sense of \eqref{eq:left positive divergence}. Then the ladder chain $\Mg=(\Mgn)_{n\ge 0}$ possesses an essentially unique stationary measure $\pi^{>}=(\pi_{i}^{>})_{i\in\cS}$, defined by \eqref{eq:pig first identity} with $c$ as above
and has $\pi$-density $f(i)=\Prob_{\xi}(\sigma_{0}=0|M_{0}=i)$ for $i\in\cS$. Moreover, the ladder chain is recurrent on $\cS^{>}:=\{i\in\cS:\pi_{i}^{>}>0\}$, and
\begin{equation}\label{Mg ergodic}
\Prob_{i}\left(\tg(\cS^{>})<\infty\right)=1
\end{equation}
for all $i\in\cS$, where $\tg(\cS^{>}):=\inf\{n\ge 1:\Mgn\in\cS^{>}\}$. Finally, positive recurrence holds iff
\begin{equation}\label{eq:condition positive recurrence}
\Erw_{\xi}\left(\frac{1}{\sigma_{1}-\sigma_{0}}\right)\ =\ \Prob_{\xi}(\sigma_{0}=0)\ <\ \infty.
\end{equation}
\end{Theorem}

\begin{proof}
If \eqref{eq:condition positive recurrence} is valid, then $\pi^{>}$ forms again a probability distribution. If the condition fails, then observe that
$$ \pi_{i}^{>}\ =\ \Prob_{\xi}(M_{0}=i,\sigma_{0}=i)\ \le\ \Prob_{\xi}(M_{0}=i)\ \le\ \pi_{i} $$
for all $i\in\cS$ that $\pi^{>}$ is still a $\sigma$-finite measure. After these observations 
the proof follows exactly the same lines as for Theorem \ref{thm:first result} and can therefore be omitted.\qed
\end{proof}

Since the dual chain $({}^{\#}M_{n},{}^{\#}X_{n})_{n\ge 0}$  with kernel ${}^{\#}Q$ given by \eqref{eq:dual kernel} is still well-defined, we see that the assertions of Corollary \ref{cor:first result} remain true as well when defining $c:=1$ in the case when \eqref{eq:condition positive recurrence} fails and thus $\Prob_{\pi}\left({}^{\#}\sle=\infty\right)=\infty$.

\section{A counterexample}\label{sec:IPFC}

The following counterexample, taken from \cite{AlsBuck:16} and discussed in greater detail there, shows that $(S_{n})_{n\ge 0}$ and $({}^{\#}S_{n})_{n\ge 0}$ may be of different fluctuation type.

\vspace{.1cm}
Let $(M_{n})_{n\ge 0}$ be a Markov chain on the set $\N_{0}$ of nonnegative integers which, when in state 0, picks an arbitrary $i\in\N$ with positive probability $p_{0i}$ and jumps back to 0, otherwise, thus $p_{i0}=1$. If we figure the $i\in\N$ being placed on a circle around 0, the transition diagram of this chain looks like a \emph{flower with infinitely many petals}, each of the petals representing a transition from 0 to some $i$ and back. With all $p_{0i}$ being positive, the chain is clearly irreducible and positive recurrent with stationary probabilities $\pi_{0}=\frac{1}{2}$ and
$$ \pi_{i}\ =\ \frac{1}{2}\,\Erw_{0}\left(\sum_{n=0}^{\tau(0)-1}\1[M_{n}=i]\right)\ =\ \frac{1}{2}\,\Prob_{0}(M_{1}=i)\ =\ \frac{p_{0i}}{2}. $$
In fact, under $\Prob_{0}$, the chain consists of independent random variables which are 0 for even $n$ and i.i.d. for odd $n$ with common distribution $(p_{0i})_{i\ge 1}$.

Next, we define the $X_{n}$ by
$$ X_{n}\ :=\ 
\begin{cases}
\hfill -p_{0i}^{-1},&\text{if }M_{n-1}=0,\,M_{n}=i,\\
2+p_{0i}^{-1},&\text{if }M_{n-1}=i,\,M_{n}=0
\end{cases}
$$
for $n\ge 1$, i.e.\ $F_{0i}=\delta_{-p_{0i}^{-1}}$ and $F_{i0}=\delta_{2+p_{0i}^{-1}}$. It follows that
$$ S_{n}\ :=\ 
\begin{cases}
n-1-p_{0M_{n}}^{-1},&\text{if }n\text{ is odd},\\
\hfill n,&\text{if }n\text{ is even}
\end{cases}
\quad\Prob_{0}\text{-a.s.},
$$
and thereupon that $(M_{n},S_{n})_{n\ge 0}$ is oscillating, for
\begin{align*}
&\hspace{1cm}\lim_{n\to\infty}\frac{S_{2n}}{2n}\ =\ 1
\shortintertext{and}
\liminf_{n\to\infty}S_{2n+1}\ &=\ \liminf_{n\to\infty}\left(n-1-\frac{1}{p_{0M_{2n+1}}}\right)\ =\ -\infty\quad\Prob_{0}\text{-a.s.}
\end{align*}
The last assertion follows from the fact that, for any $a>0$,
$$ \sum_{n\ge 0}\Prob_{0}\left(\frac{1}{p_{0M_{2n+1}}}>an\right)\ =\ \sum_{n\ge 0}\Prob_{0}(X_{1}>an)\ \ge\ \frac{\Erw_{0}X_{1}}{a}\ =\ \infty $$
and an appeal to the Borel-Cantelli lemma, giving
\begin{equation}\label{eq:1/p_{0i}>n i.o.}
\Prob_{0}\left(\frac{1}{p_{0M_{2n+1}}}>an\text{ i.o.}\right)\ =\ 1.
\end{equation}

\vspace{.2cm}
Turning to the dual MRW $({}^{\#}M_{n},{}^{\#}S_{n})_{n\ge 0}$, it has increments
$$ {}^{\#}X_{n}\ :=\ 
\begin{cases}
2+p_{0i}^{-1},&\text{if }{}^{\#}M_{n-1}=0,\,{}^{\#}M_{n}=i,\\
\hfill -p_{0i}^{-1},&\text{if }{}^{\#}M_{n-1}=i,\,{}^{\#}M_{n}=0
\end{cases}
$$
for $n\ge 1$ and is therefore positive divergent, for
$$ {}^{\#}S_{n}\ :=\ 
\begin{cases}
n+1+p_{0M_{n}}^{-1},&\text{if }n\text{ is odd},\\
\hfill n,&\text{if }n\text{ is even}
\end{cases}
\quad\Prob_{0}\text{-a.s.}
$$

\bibliographystyle{abbrv}
\bibliography{StoPro}

\end{document}